\documentclass[10pt,reqno,a4]{amsart}
\usepackage[margin=2.5cm]{geometry}
\usepackage[utf8]{inputenc}
\DeclareUnicodeCharacter{00A0}{~}
\usepackage{amssymb,amsmath,amscd,amsfonts,amsthm,bbm}
\usepackage[usenames, dvipsnames]{color}
\usepackage{graphicx}
\usepackage{caption}
\usepackage{subcaption}
\usepackage{enumitem}
\usepackage[colorlinks=true]{hyperref}
\usepackage{array}
\newtheorem{theo}{Theorem}[section]
\newtheorem{heuristics}{Heuristics}[section]

\newtheorem{lemma}[theo]{Lemma}

\newtheorem{prop}[theo]{Proposition}

\theoremstyle{remark}
\newtheorem{rem}[theo]{Remark}

\theoremstyle{definition}

\newcommand{\pij}{\partial_{ij}}
\newcommand{\pik}{\partial_{ki}}
\newcommand\nc\newcommand
\newcommand\dmo\DeclareMathOperator

\nc{\HH}{{\mathcal H}}
\nc{\x}{\bs{x}}
\nc{\ones}{\bs{1}}
\nc{\ee}{\bs{e}}
\nc{\J}{{\mathcal J}}
\nc{\lplus}{\boldsymbol{\ell}^+}

\nc{\Rtilde}{{\widetilde R}_k}
\nc{\dz}{{\bf d}_z}
\nc{\N}{\mathbb{N}}
\nc{\R}{\mathbb{R}}
\nc{\C}{\mathbb{C}}
\nc{\E}{\mathbb{E}}
\nc{\PP}{\mathbb{P}}
\nc{\bdm}{\begin{displaymath}}
\nc{\edm}{\end{displaymath}}
\nc{\bea}{\begin{eqnarray*}}
\nc{\eea}{\end{eqnarray*}}
\nc{\la}{\langle}
\nc{\ra}{\rangle}
\nc{\Cplus}{\mathbb{C}_+}
\nc{\Rplus}{\mathbb{R}^+}
\nc{\pitilde}{\tilde{\pi}}
\nc{\tv}{\mathrm{tv}}
\nc{\Cnabla}{\mathbb{C}^{\nabla}}
\nc{\im}{\mathrm{Im}}
\nc{\bs}{\boldsymbol}
\nc{\ti}{\tilde}
\nc{\Msub}{M_{\mathrm{sub}}}
\nc{\diag}{\mathrm{diag}}
\nc{\ii}{\mathrm{i}}%{\boldsymbol{i}}
\numberwithin{equation}{section}
\nc{\bv}{\boldsymbol{\varepsilon}}
\nc{\dzz}{\boldsymbol{\delta}_z}
\nc{\tr}{\mathrm{tr}\,}
\nc{\cvgP}[1]{\xrightarrow[#1]{\mathcal P}}
\nc{\cvgD}{\xrightarrow[]{\mathcal D}}
\nc{\eqdef}{\stackrel{\triangle}{=}} 
\nc{\lrn}{\left|\!\left|\!\left|} 
\nc{\rrn}{\right|\!\right|\!\right|} 
\nc{\bell}{\boldsymbol{\ell}}
\nc{\blambdaR}{\boldsymbol{\lambda}_R}
\nc{\comment}[1]{\textcolor{blue}{[{\it #1}]}}

\nc{\rhoVn}{\rho(V_n)}	%{\rho_{{\bf V}_n}}
\nc{\rhoV}{\rho(V)}	%{\rho_{\bf V}}
\nc{\tkappa}{{\tilde \kappa}}

\nc{\Ncenter}{\stackrel{\circ}{\mathcal N}_n}
\nc{\Sn}{{\mathcal S}_n}
\nc{\Kn}{{\mathcal K}_n}

\dmo{\Imm}{Im}
\dmo{\Real}{Re}
\dmo*{\dist}{dist}
\dmo{\var}{var}
\dmo{\trace}{Tr}
\dmo{\rank}{rank}

\nc{\Rd}{R_n^{1/2}}
\nc{\Rdb}{\bar{R}_n^{1/2}}
\nc{\V}{\mathcal V}
\nc{\VV}{|\mathcal V|^2}
\nc{\dlp}{d_{\mathcal LP}}
\nc{\Qij}{Q_{[ij]}}

\nc{\cU}{{\mathcal U}}

\nc{\smax}{\sigma_{\max}} %{\boldsymbol{\sigma}_{\boldsymbol{\max}}}
\nc{\smin}{\sigma_{\min}} %{\boldsymbol{\sigma}_{\boldsymbol{\min}}}

\nc{\Xn}{X^{\mathcal N}}
\nc{\Yn}{Y^{\mathcal N}}
\nc{\Hn}{H^{\mathcal N}}
\nc{\Gn}{G^{\mathcal N}}
\nc{\tGn}{\widetilde{G}^{\mathcal N}}
\nc{\Fn}{F^{\mathcal N}}
\nc{\cFn}{{F}^{'\mathcal N}}

\nc{\vxi}{\vec{\xi}}
\nc{\vxin}{\vec{\xi}^{\mathcal N}}
\nc{\xin}{\xi^{\mathcal N}}
\nc{\bxin}{\bar{\xi}^{\mathcal N}}

\nc{\Ec}{\E_{\{1\}}}

\nc{\Oeta}[1]{{\mathcal O}_{\eta} \left( {#1}\right)}
\nc{\vOeta}[1]{\vec{\mathcal O}_{\eta} \left( {#1}\right)}
\nc{\specnorm}[1]{\left\| {#1}\right\|_{\mathrm{sp}}}
\nc{\posneq}{\succcurlyeq_{\neq}}

\nc{\bq}{\bs{q}}
\nc{\qt}{\widetilde{q}}
\nc{\bqt}{\bs{\qt}}
\nc{\qvec}{\vec{\bs q}}
\nc{\qvecstar}{\qvec_*}
\nc{\bqstar}{\bq_*}
\nc{\bqtstar}{\bqt_*}

\nc{\rr}{r}
\nc{\rt}{\widetilde{r}}
\nc{\br}{\bs{r}}
\nc{\brt}{\bs{\rt}}
\nc{\rvec}{\vec{\bs r}}
\nc{\rvecstar}{\rvec_*}
\nc{\brstar}{\br_*}
\nc{\brtstar}{\brt_*}

\nc{\pt}{\widetilde{p}}
\nc{\bp}{\bs{p}}
\nc{\bpt}{\bs{\pt}}
\nc{\pvec}{\vec{\bp}}

\nc{\vp}{\varphi}
\nc{\vpt}{\widetilde{\varphi}}
\nc{\bphi}{\bs{\varphi}}
\nc{\bphit}{\bs{\widetilde \varphi}}
\nc{\phivec}{\bs{\vec{\varphi}}}

\nc{\q}{\bs{q}}
\nc{\qtilde}{\bs{\widetilde q}}	
\nc{\p}{\bs{p}}
\nc{\ptilde}{\bs{\widetilde p}}
\nc{\tp}{\widetilde{\varphi}}
\nc{\bphitilde}{\bs{\widetilde \varphi}}
\nc{\bphivec}{\bs{\vec{\varphi}}}

\nc{\vecepsilon}{\vec{\boldsymbol{\varepsilon}}}

\nc{\Qa}{{\mathcal Q}(\bs{\alpha}, A, \bs{a})}
\nc{\Qb}{{\mathcal Q}(\bs{\alpha}, B, \bs{b})}

\nc{\Sab}{{\mathcal S}_{AB}}
\nc{\Sa}{{\mathcal S}_{A}}
\nc{\Sb}{{\mathcal S}_{B}}

\nc{\Sig}{A}

\dmo{\eps}{\varepsilon}
\dmo{\ls}{\lesssim}
\dmo{\gs}{\gtrsim}

\nc{\expo}[1]{\exp \left( #1 \rule{0mm}{3mm}\right)}
\dmo{\e}{\mathbb{E}}
\dmo{\pr}{\mathbb{P}}
\dmo{\un}{\mathbbm{1}}
\dmo{\1}{\mathbf{1}}
\nc{\tran}{\mathsf{T}} 	
\nc{\scut}{\snot}		%{\boldsymbol{\sigma}_{\mathrm{cut}}}
\nc{\snot}{\sigma_0}	%{\bs{\sigma}_0}
\nc{\mN}{\mathcal{N}}
\nc{\vpmax}{\|\bphi\|_\infty}%{\varphi_{\max}}
\nc{\tpmax}{\|\bphitilde\|_\infty}
\nc{\bS}{\bs{S}}
\nc{\bd}{\bs{d}}
\nc{\bdt}{\bs{\tilde{d}}}
\nc{\dt}{\tilde{d}}
\nc{\HS}{\mathsf{HS}}
\nc{\Mo}{M_0}
\nc{\Res}{\bs{R}}
\nc{\Y}{\bs{Y}}
\nc{\Am}{\bs{A}^{(m)}}
\nc{\Vm}{\bs{V}^{(m)}}
\nc{\Ym}{\bs{Y}^{(m)}}
\nc{\Lm}{\check{\bs{L}}^{(m)}}
\nc{\wY}{{\widetilde Y}}
\nc{\wA}{{\widetilde A}}

\nc{\Blue}[1]{\textcolor{blue}{#1}}
\nc{\comN}[1]{\textcolor{ForestGreen}{#1}}

\newcommand{\Rr}{R^{\br}}

\title[ Positive solutions for large random linear systems]{Positive solutions for large random linear systems}

\author[P. Bizeul, J. Najim]{Pierre Bizeul, Jamal Najim} 

\date{\today}


\begin{document}

\begin{abstract}  Consider a large linear system where $A_n$ is a $n\times n$ matrix with independent real standard Gaussian entries, $\ones_n$ is a $n\times 1$ vector of ones and with unknown the $n\times 1$ vector $\x_n$ satisfying
$$
\x_n = \ones_n +\frac 1{\alpha_n\sqrt{n}} A_n \x_n\, .
$$
We investigate the (componentwise) positivity of the solution $\x_n$ depending on the scaling factor $\alpha_n$ as the dimension $n$ goes to $\infty$. We prove that there is a sharp phase transition at the threshold $\alpha^*_n =\sqrt{2\log n}$: below the threshold ($\alpha_n\ll \sqrt{2\log n}$), $\x_n$ has negative components with probability tending to 1 while above ($\alpha_n\gg \sqrt{2\log n}$), all the vector's components are eventually positive with probability tending to 1. At the critical scaling $\alpha^*_n$, we provide a heuristics to evaluate the probability that $\x_n$ is positive.

Such linear systems arise as solutions at equilibrium of large Lotka-Volterra systems of differential equations, widely used to describe large biological communities with interactions such as foodwebs for instance. 

In the domaine of positivity of the solution $\x_n$, that is when $\alpha_n\gg \sqrt{2\log n}$, we establish that the Lotka-Volterra system of differential equations whose solution at equilibrium is precisely $\x_n$ is stable in the sense that its jacobian
$$
\J(\x_n) = \diag(\x_n)\left(-I_n + \frac {A_n}{\alpha_n\sqrt{n}}\right)
$$
has all its eigenvalues with negative real part with probability tending to one.

Our results shed a new light and complement the understanding of feasibility and stability issues for large biological communities with interaction.

\end{abstract}

\maketitle

%\tableofcontents

%%%%%%%%%%%%
%%%%%%%%%%%% \input{sec1_intro.tex}
%%%%%%%%%%%%

\section{Introduction} 
\label{sec:intro} Denote by $A_n$ a $n\times n$ matrix with independent Gaussian ${\mathcal N}(0,1)$ entries and by $\alpha_n$ a positive sequence. We  
are interested in the componentwise positivity of the $n\times 1$ vector $\x_n$, solution of the linear system
\begin{equation}\label{eq:equilibrium}
\x_n = \ones_n + \frac{1}{\alpha_n\sqrt{n}} A_n \x_n\ ,
\end{equation}
where $\ones_n$ is the $n\times 1$ vector with components 1. 

It is well-known since Geman \cite{geman-1986} that the spectral radius of $\frac{A_n}{\sqrt{n}}$ almost surely (a.s.) converges to 1, so that matrix $\left(I_n - \frac{A_n}{\alpha_n\sqrt{n}} \right)$ is eventually invertible as long as $\alpha_n\gg1$. In this case, vector $\x_n=(x_k)_{k\in[n]}$ where $[n]=\{1,\cdots, n\}$ writes
$$
\x_n = \left(I_n - \frac{A_n}{\alpha_n\sqrt{n}} \right)^{-1}\ones_n\quad \textrm{with}\quad x_k=\ee_k^* \left(I_n - \frac{A_n}{\alpha_n\sqrt{n}} \right)^{-1}\ones_n\ ,
$$
where $\ee_k$ is the $n\times 1$ canonical vector and $B^*$ denotes the transconjugate of matrix $B$ (or simply its transpose if $B$ is real). 

The positivity of the $x_k$'s is a key issue in the study of Large Lotka-Volterra systems, widely used in mathematical biology and ecology to model populations with interactions. 

Consider for instance a given foodweb and denote by $\x_n(t)=(x_k(t))_{k\in [n]}$ the vector of abundances of the various species within the foodweb at time $t$. A standard way to connect the various abundances is via a Lotka-Volterra (LV) system of equations that writes 
\begin{equation}\label{eq:LV}
\frac{dx_k(t)}{dt} = x_k(t)\, \left( 1 - x_k(t) + \frac{1}{\alpha_n \sqrt{n}} \sum_{\ell\in[n]} A_{k\ell} x_{\ell}(t)\right)\qquad \textrm{for}\quad k\in [n]\, ,
\end{equation}
where the interactions $(A_{k\ell})$ can be modeled as random in the absence of any prior information. Here, the $A_{k\ell}$'s are assumed to be i.i.d. ${\mathcal N}(0,1)$.
At the equilibrium $\frac{d\x_n}{dt}=0$, the abundance vector $\x_n$ is solution of \eqref{eq:equilibrium} and a key issue is the existence of a {\em feasible} solution, that is a solution $\x_n$ where all the $x_k$'s are positive. Dougoud et al. \cite{dougoud2018feasibility}, based on Geman and Hwang \cite{geman-1982}, proved that a feasible solution is very unlikely to exist if $\alpha_n\equiv\alpha$ is a constant. In fact, the CLT proved in \cite{geman-1982} asserts that for any fixed number $M$ of components
$$
\left(x_k-1\right)_{k\in[M]}\quad \xrightarrow[n\to\infty]{\quad \mathcal D\quad} \quad Z\ \sim\ {\mathcal N}(0,\sigma^2_{\alpha}\, I_M)\ ,
$$
where $\xrightarrow[]{\mathcal D}$ (resp. $\xrightarrow[]{\mathcal P}$) stands for the convergence in distribution (resp. in probability) and where $\sigma_{\alpha}^2 ={\mathcal O}(1)$. As an important consequence, vectors $\x_n$ with positive components will become extremely rare since
$$
\PP\{ x_k>0,\, k\in [M]\} \xrightarrow[n\to\infty]{} \left( \int_{\sigma_\alpha^{-1}}^\infty \frac{e^{-x^2/2}}{\sqrt{2\pi}} \, dx\right)^M \quad \Rightarrow\quad 
\PP \{ x_k>0,\, k\in [n]\} \xrightarrow[n\to\infty]{} 0\, .
$$

In this article, we consider a growing scaling factor $\alpha_n\to\infty$ and study the positivity of $\x_n$'s components in relation with $\alpha_n$. We find that there exists a critical threshold $$\alpha^*_n = \sqrt{2\log n}$$ below which feasible solutions exist with vanishing probability and above which feasible solutions are more and more likely to exist. More precisely, we prove the following:
\begin{theo}[Feasibility]\label{th:main}
Let $\alpha_n \xrightarrow[n\to\infty]{} \infty$  and denote by $\alpha_n^*=\sqrt{2\log n}$. Let $\x_n=(x_k)_{k\in [n]}$ be the solution of \eqref{eq:equilibrium}. 
\begin{enumerate}
\item If there exists $\varepsilon>0$ such that eventually $\alpha_n\le (1-\varepsilon) \alpha_n^*$ then 
$$
\mathbb{P}\left\{ \min_{k\in [n]} x_k>0\right\} \xrightarrow[n\to\infty]{} 0\, .
$$

\item If there exists $\varepsilon>0$ such that eventually $\alpha_n\ge (1+\varepsilon)\alpha_n^*$ then

$$
\mathbb{P}\left\{ \min_{k\in [n]} x_k>0\right\} \xrightarrow[n\to\infty]{} 1\, .
$$
\end{enumerate}
\end{theo}
Proof of Theorem \ref{th:main} is based on an analysis of the order of magnitude of the extreme values of the $x_k$'s, which relies on Gaussian concentration of Lipschitz functionals whose argument is matrix $A_n$. 

\begin{rem} In Figure \ref{fig:phase-transition}, we illustrate the transition toward feasibility depending on the scaling $\alpha_N(\kappa)=\kappa\sqrt{\log(N)}$. 
For $\kappa\in [0.5, 2.5]$, we plot the proportion of feasible solutions $\x_N(\kappa)$ obtained after 500 simulations. The transition occurs at the optimal scaling 
$\alpha_N^*=\sqrt{2\log(N)}$ corresponding to $\kappa=\sqrt{2}$.
\end{rem}

\begin{rem}
Notice that the convergence of $\frac 1{\alpha_n^*}$ to zero is extremely slow, as shown in Table \ref{tab:slow}, and could easily be mistaken with some constant scaling $\sigma<1$ where $\sigma = \frac 1{\alpha_n^*}$. 
\end{rem}

\begin{figure}[ht] 
\centering
  \includegraphics[width=\linewidth]{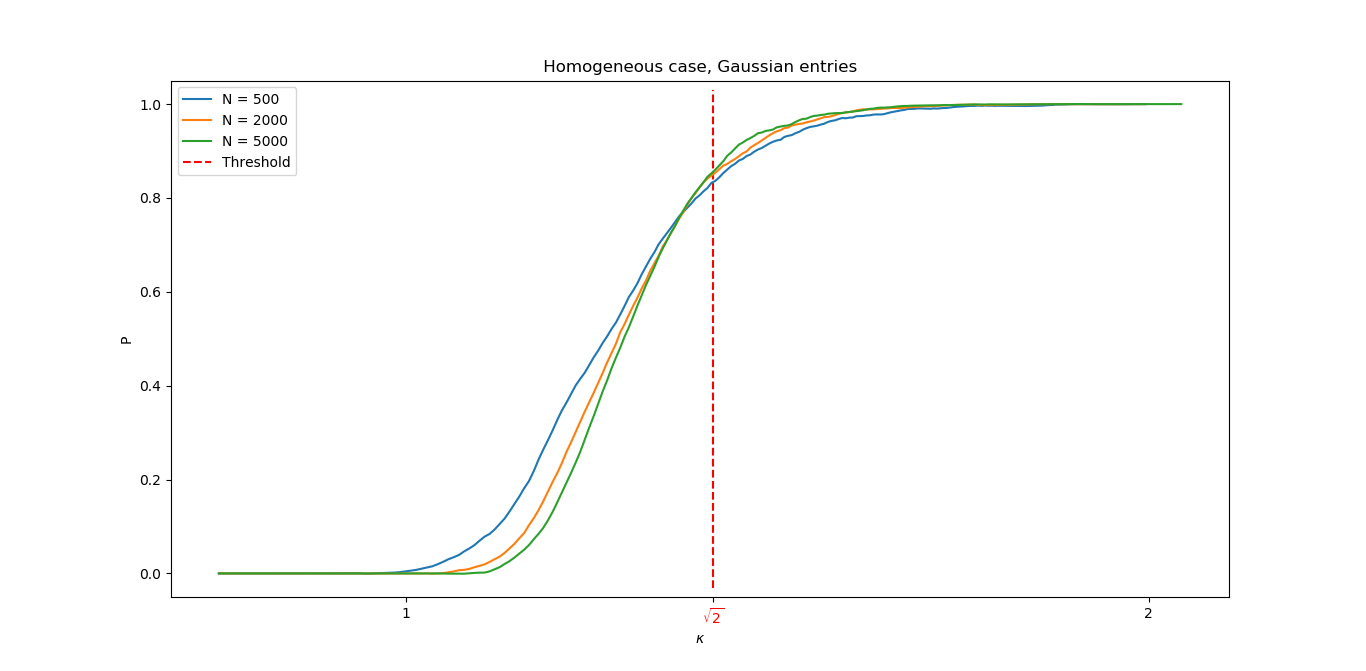}
 \caption{Transition toward feasibility. We consider different values of $N$, respectively 500 (blue), 2000 (yellow), 5000 (green). For each $N$ and each $\kappa$ on the $x$-axis, we simulate 500 $N\times N$ matrices $A_N$ and compute the solution $\x_N$ of \eqref{eq:equilibrium} at the scaling 
$\alpha_N(\kappa)=\kappa\sqrt{\log(N)}$. Each curve represents the proportion of feasible solutions $\x_N$ obtained for 500 simulations and has been smoothed by a Savistky-Golay filter. The red dotted vertical line corresponds to the critical scaling $\alpha_N^*=\sqrt{2\log(N)}$ for $\kappa=\sqrt{2}$. The proportion of feasible solutions ranges from 0 for $\kappa\le 1$ to 1 for $\kappa\ge2$.} 
\label{fig:phase-transition}
\end{figure}

To complement the picture, we provide the following heuristics at the critical scaling $\alpha_n^*=\sqrt{2\log n}$:
\begin{equation}\label{eq:heuristics}
\PP\left\{ \min_{k\in [n]} x_k>0\right\} \quad \approx\quad  1 - \sqrt{\frac {e}{4\pi \log n}}
+\frac{e}{8\pi\log n}\qquad \textrm{as}\quad n\to\infty\, .
\end{equation}

\renewcommand{\arraystretch}{1.8}
\begin{table}[!h]
\caption{The quantity $\frac 1{\alpha_n^*}=\frac 1{\sqrt{2\log n}}$ vanishes extremely slowly as $n$ increases.}
\begin{tabular}{|l||c|c|c|c|c|c|}
\hline
$n$ & $10^2$ & $10^3$ & $10^4$ & $10^5$ & $10^6$\\
\hline
$\frac 1{\alpha_n^*}$ & 0.33 & 0.27 & 0.23 & 0.21 & 0.19\\
\hline
\end{tabular}
\label{tab:slow}
\end{table}

Aside from the question of feasibility arises the question of {\em stability} : for a complex system, how likely a perturbation of the solution $\x_n$ at equilibrium  
will return to the equilibrium? Gardner and Ashby \cite{gardner1970connectance} considered stability issues of complex systems connected at random. Based on the circular law for large matrices with i.i.d. entries, May \cite{may1972will} provided a complexity/stability criterion and motivated the systematic use of large random matrix theory in the study of foodwebs, see for instance  Allesina et al. \cite{Allesina:2015ux}. Recently, Stone \cite{stone2018feasibility} and Gibbs {\it et al.} \cite{gibbs2018effect} revisited the relation between feasibility and stability.

We complement the information of Theorem \ref{th:main} by adressing the question of stability in the context of a Lotka-Volterra system \eqref{eq:LV} and prove that under the first condition of the theorem feasibility and stability occur simultaneously.

Recall that the solution at equilibrium $\x_n$ is stable if the Jacobian matrix ${\mathcal J}$ of the Lotka-Volterra system evaluated at $\x_n$, that is
\begin{equation}\label{eq:jacobian}
\J(\x_n) = \diag(\x_n)\left(-I_n + \frac {A_n}{\alpha_n\sqrt{n}}\right)
\end{equation}
has all its eigenvalues with negative real part.

\begin{theo}[Stability]\label{th:stability} Let $\x_n=(x_k)_{k\in [n]}$ be the solution of \eqref{eq:equilibrium}.  Denote by $\lplus=\limsup_{n\to \infty}\frac{\sqrt{2\log n}}{\alpha_n}$ and assume that $\lplus<1$. Denote by ${\mathcal S}_n$ the spectrum of $\J(\x_n)$ and let $\lambda \in {\mathcal S}_n$. Then 
$$
\max_{\lambda\in {\mathcal S}_n}  \min_{k\in [n]} \left| \lambda + x_k\right| \ \xrightarrow[n\to\infty]{\mathcal P}\ 0\, .
$$
Moreover,
\begin{equation}\label{eq:stability}
\max_{\lambda\in {\mathcal S}_n}\mathrm{Re}(\lambda) \ \le\  -(1-\lplus) + o_P(1)\, .
\end{equation}
\end{theo}

Proof of Theorem \ref{th:stability} relies on standard perturbation results from linear algebra and on Theorem \ref{th:main}.

\subsection*{Organization of the paper} Proof of Theorem \ref{th:main} is provided in Section \ref{section:proof-main}. Theorem \ref{th:stability} is proved in Section \ref{section:proof-stability}. In Section \ref{section:various}, elements to bear out heuristics \eqref{eq:heuristics} are provided. We also formulate some concluding remarks for non-homogeneous linear systems where vector $\ones_n$ is replaced by a positive vector $\br_n$ 
and briefly mention possible extensions to non-Gaussian entries.

\subsection*{Acknowlegments} JN thanks Christian Mazza for introducing him to the study of large LV systems in theoretical ecology. The authors thank François Massol and Olivier Guédon for fruitful discussions. 

\section{Positive solutions: proof of Theorem \ref{th:main}} \label{section:proof-main}

We will use the following notations for the various norms at stake: if $\boldsymbol{v}$ is a vector then $\|\boldsymbol{v}\|$ stands for its euclidian norm; if $A$ is a matrix then $\|A\|$ stands for its spectral norm and $\|A\|_F=\sqrt{\sum_{ij} |A_{ij}|^2}$ for its Frobenius norm. Let $\varphi$ be a function from $\Sigma=\R$ or $\C$ to $\C$ then $\|\varphi\|_{\infty} =\sup_{x\in \Sigma} |\varphi(x)|$. 

\subsection{Some preparation and strategy of the proof}

Denote by $Q_n=\left( I_n - \frac {A_n}{\alpha_n\sqrt{n}}\right)^{-1}
$ the resolvent and by $s(B)$ the largest singular value of a given matrix $B$. Then it is well known that almost surely $s_n:=s(n^{-1/2} A_n)  \xrightarrow[n\to\infty]{} 2$ (see for instance \cite[Chapter 5]{book-bai-silverstein}) hence $s\left(\frac 1{\alpha_n\sqrt{n}} A_n\right) \xrightarrow[n\to\infty]{} 0$. In particular, the solution 
$$
\x_n =\left( x_k\right)_{k\in [n]} = \left( I_n - \frac {A_n}{\alpha_n\sqrt{n}}\right)^{-1} \ones_n = Q_n\, \ones_n\ ,
$$
with $I_n$ the $n\times n$ identity, is uniquely defined almost surely. In order to study the minimum of $\x_n$'s components, we partially unfold the above resolvent (in the sequel, we will simply denote $A, \alpha,\ones,Q$ instead of $A_n, \alpha_n, \ones_n,Q_n$) and write:
\begin{eqnarray}
x_k&=& \ee_k^* \x\quad = \quad \ee_k^*Q\, \ones \quad=\quad  \sum_{\ell=0}^\infty \ee^*_k \left( \frac A{\alpha\sqrt{n}}\right)^\ell \ones\, , \nonumber\\ 
&=&  1+ \frac 1{\alpha} \ee_k^* \left(n^{-1/2} A \right) \ones  + \frac 1{\alpha^2} \ee_k^* \left( n^{-1/2} A\right)^2 Q\, \ones\, \quad =\quad  1+\frac 1{\alpha} Z_k +\frac 1{\alpha^{2}} R_k\ , \label{eq:decomp}
\end{eqnarray}
where 
\begin{equation}\label{def:Gaussian-Remainder}
Z_k=\ee_k^* \left(n^{-1/2} A \right) \ones= \frac 1{\sqrt{n}} \sum_{i=1}^n A_{ki}
\qquad \textrm{and}\qquad R_k= \ee_k^* \left( n^{-1/2} A\right)^2 Q\, \ones\, .
\end{equation} 
Notice in particular that the $Z_k$'s are i.i.d. standard Gaussian. Before focusing on the analysis of the remaining term $R_k$, we recall standard results for extreme values of Gaussian random variables.

\subsubsection*{Extreme values of Gaussian random variables}
Consider the sequence $(Z_k)$ of standard Gaussian i.i.d. random variables and let
\begin{equation}\label{def:beta}
M_n=\max_{k\in [n]} Z_k\, , \quad \check M_n =\min_{k\in[n]} Z_k\, , \quad 
\alpha^*_n= \sqrt{2\log n}\qquad \textrm{and}\qquad \beta^*_n = \alpha^*_n - \frac 1{2\alpha^*_n}  \log ( 4\pi \log n)\ .
\end{equation}
Denote by $G(x)= e^{-e^{-x}}$ the cumulative distribution of a Gumbel distributed random variable. 

Then the following results are standard, see for instance \cite[Theorem 1.5.3]{leadbetter2012extremes}: for all $x\in \R$
\begin{eqnarray}
\mathbb{P} \left\{ \alpha^*_n( M_n-\beta^*_n) \le  x \right\} &\xrightarrow[n\to\infty]{}& G(x)\, ,\label{eq:MAX}\\
\mathbb{P} \left\{ \alpha^*_n(\check M_n +\beta^*_n) \ge -x \right\} &\xrightarrow[n\to\infty]{}& G(x)\label{eq:MIN}\, .
\end{eqnarray}

\subsubsection*{Strategy of the proof}
Eq. \eqref{eq:decomp} immediatly yields 
$$
\left\{ 
\begin{array}{lcl}
\min_{k\in[n]} x_k& \ge & 1+\frac 1{\alpha} \check M +\frac 1{\alpha^2} \min_{k\in[n]} R_k\ ,\\
\min_{k\in[n]} x_k& \le & 1+\frac 1{\alpha} \check M +\frac 1{\alpha^2} \max_{k\in[n]} R_k\, .
\end{array}\right.
$$
We rewrite the first equation as
\begin{equation}\label{eq:useful}
\min_{k\in[n]} x_k \ \ge\   1+\frac {\alpha^*_n}{\alpha_n}\left( \frac{\check M +\beta^*_n}{\alpha^*_n} - \frac{\beta^*_n}{\alpha^*_n} +\frac {\min_{k\in[n]} R_k}{\alpha^*_n \alpha_n} \right)
\ =\ 1+\frac{\alpha^*_n}{\alpha_n} \left( -1+o_P(1) +\frac {\min_{k\in[n]} R_k}{\alpha^*_n \alpha_n} \right)\, ,
\end{equation}
where we have used the fact that $(\alpha^*_n)^{-1} (\check M + \beta^*_n)=o_P(1)$. Similarly,
$$
\min_{k\in[n]} x_k \ \le\ 1+\frac{\alpha^*_n}{\alpha_n} \left( -1+o_P(1) +\frac {\max_{k\in[n]} R_k}{\alpha^*_n \alpha_n} \right)\, .
$$
The theorem will then follow from the following lemma.

\begin{lemma}\label{lemma:main} The following convergence holds
$$
\frac{\max_{k\in [n]} R_k}{\alpha_n \sqrt{2\log n}}\  \xrightarrow[n\to\infty]{\mathcal P} \ 0\qquad \textrm{and}\qquad 
\frac{\min_{k\in [n]} R_k}{\alpha_n \sqrt{2\log n}} \ \xrightarrow[n\to\infty]{\mathcal P} \ 0 \ .
$$
\end{lemma}
Proof of Lemma \ref{lemma:main} requires a careful analysis of the order of magnitude of the extreme values of the remaining term $(R_k)_{k\in[n]}$. 
It is postponed to Section \ref{sec:proof-lemma-main}.
\subsection{Lipschitz property and tightness of $R_k(A)$}
Let $\varphi:\R^+\to [0,1]$ be a smooth function with values
$$
\varphi(x) =\begin{cases}
1& \textrm{if}\ x\in [0,2+\eta]\\
0& \textrm{if}\ x\ge 3
\end{cases}\ ,
$$
and strictly decreasing from $1$ to zero as $x$ goes from $2+\eta$ to $3$.
Recall that $s_n= s(n^{-1/2} A)$ is the largest singular value of the normalized matrix $n^{-1/2} A$ and denote by 
$$
\varphi_n := \varphi (s_n) = \varphi \left( s( n^{-1/2} A ) \right)\, .
$$ 
Notice that $\mathbb{P} \{ \varphi_n<1\} =\mathbb{P}\{ s_n>2+\eta\} \xrightarrow[n\to\infty]{}0$ (as a by-product of the a.s. convergence of $s_n$ to $2$).

Instead of directly working with $R_k$ we introduce the truncated quantity 
\begin{equation}\label{eq:truncated}
\widetilde R_k = \varphi_n R_k\ .
\end{equation} 
For a given $n\times n$ matrix $A$, we may consider its $2n\times 2n$ hermitized matrix $\HH(A)$ defined as
$$
\HH(A) =\left( \begin{array}{cc} 0&A\\ A^*&0\end{array}\right)\, .
$$
Recall that the singular values of $A$ together with their opposites are the eigenvalues of $\HH(A)$.

We prove hereafter that as a function of the entries of matrix $A$, the function $A\mapsto \widetilde R_k(A)$ is lipschitz. 

\begin{lemma}\label{lem:LIP} Let $\widetilde R_k$ be given by \eqref{eq:truncated}, then the function $A\mapsto \Rtilde(A)$ is Lipschitz, i.e. 
\begin{equation}\label{eq:lipschitz}
\left| \widetilde R_k(A) - \Rtilde(B)\right| \le K \|A-B\|_F\ ,
\end{equation}
where $\|A\|_F$ is the Frobenius norm and $K$ is a constant independent from $k$ and $n$.
\end{lemma}

\begin{proof} Notice that $\varphi(s_n)=0$ and $\varphi'(s_n)=0$ for $s_n\ge 3$, which implies that one may consider the bound $s_n\le 3$  in the following computations, for $\widetilde R_k$ or its derivatives would be zero otherwise. Recall the definition of the resolvent $Q=\left( I - \frac{A}{\alpha\sqrt{n}}\right)^{-1}$
then $Q^{-1}Q = I$ which yields $Q=I +\frac A{\alpha\sqrt{n}}Q$ from which we deduce that 
$$
\varphi_n \left\| Q \right\| \ \ \le\ \ \varphi_n\left( 1-\frac 1{\alpha} \left\| n^{-\frac 12}A\right\|\right)^{-1}\quad\le\quad  
\frac 1{1- 3\alpha^{-1}}\ \ \le \ \ 3
$$
for $n$ large enough. 

We first consider a matrix $A$ such that $\HH(A)$ has simple spectrum (i.e. with $2n$ distinct eigenvalues, each with multiplicity 1). We denote by $\pij=\frac{\partial \ }{\partial A_{ij}}$ and prove that 
the vector $\nabla \Rtilde(A)=\left( \pij \widetilde R(A),\ i,j\in [n]\right)$ satisfies
\begin{equation}\label{eq:gradient}
\| \nabla \Rtilde(A)\| =\sqrt{\sum_{ij} \left| \pij \widetilde R_k(A)\right|^2}\le K\, .
\end{equation}
To lighten the notations, we may drop the dependence of $\widetilde R_k$ in $A$. We begin by computing
\begin{eqnarray*}
\pij \widetilde R_k &=&\lim_{h\to 0 } \frac{ \Rtilde(A+h\ee_i\ee_j^*) - \Rtilde(A)}h\ ,\\
&=&  (\pij \varphi_n) R_k
+ \varphi_n\, \ee_k^*  \left( \pij  \left(n^{-\frac 12} A\right)^2\right)  Q\, \ones + \varphi_n\, \ee_k^*  \left( n^{-\frac 12} A\right)^2 \left( \pij Q\right) \ones \quad =:\quad T_{1,ij} + T_{2,ij}+T_{3,ij}\, .
\end{eqnarray*}
Straightforward computations yield
$$
\pij\left( n^{-\frac 12} A\right)^2 = \frac 1n \left( A\ee_i\ee_j^* + \ee_i \ee_j^*A\right)\quad \textrm{and}
\quad \pij Q = \frac 1{\alpha \sqrt{n}} Q \ee_i \ee_j^* Q\, .
$$
It remains to compute 
$
\pij \varphi_n  =  \varphi'(s_n) \pij s_n
$.
Recall that ${\mathcal H}(A)$ has a simple spectrum and notice that $A\mapsto s_n(A)$ is differentiable. In fact, since $s_n$ is simple, it is a simple root of the characteristic polynomial. In particular, it is not a root of its derivative and one can use the implicit function theorem to conclude.
Let $\bs{u}$ and $\bs{v}$ be respectively the left and right normalized singular vectors associated to $s(A)$. Then
$$
\HH(A)\bs{w} = s(A) \bs{w}\quad \textrm{with}\quad \bs{w}=\begin{pmatrix} \bs{u}\\ \bs{v}\end{pmatrix}\quad\textrm{and}\quad 
\|\bs{w}\|^2=2\ , 
$$
moreover $\bs{w}$ is (up to a scaling factor) the unique eigenvector of $s(A)$ since $s(A)$ has multiplicity one by assumption. We can now apply \cite[Theorem 6.3.12]{book-horn-johnson} to compute $s_n$'s derivative:
\begin{equation}\label{eq:sn-derivative}
\pij s(A) =  \frac 1{\|\bs{w}\|^2} \left( \bs{u}^* \ee_i \ee_j^* \bs{v} +\bs{v}^* \ee_j \ee_i^* \bs{u}\right) = \bs{u}^* \ee_i \ee_j^* \bs{v} \qquad \textrm{hence}\qquad 
\pij s_n = \frac 1{\sqrt{n}} \bs{u}^* \ee_i \ee_j^* \bs{v} 
\end{equation}
(recall that all the considered vectors are real). 

We first handle the term $T_{1,ij}$.
\begin{eqnarray*}
\sum_{ij} \left| T_{1,ij}\right|^2 &=& \sum_{ij} \left| \bs{u}^* \ee_i \ee_j^* \bs{v} \varphi'(s_n) \ee_k^*  \left( n^{-1/2} A\right)^2 Q \frac {\ones}{\sqrt{n}}
\right|^2\ ,\\
&\le &  3^6 \|\varphi'\|^2_{\infty} \sum_i \left|\bs{u}^* \ee_i\right|^2 \sum_j \left| \ee_j^* \bs{v}\right|^2  \quad \le\quad 3^6 \|\varphi'\|^2_{\infty}\, . 
\end{eqnarray*}
We now handle the term $T_{2,ij}$.
\begin{eqnarray*}
\sum_{ij} \left| T_{2,ij}\right|^2 &=& \sum_{ij} \left|\varphi_n  \ee_k^* \left( \frac A{\sqrt{n}} \ee_i \ee_j^* +  \ee_i \ee_j^* \frac A{\sqrt{n}} \right) 
Q \frac {\ones}{\sqrt{n}}
\right|^2\ ,\\
&\le& 2\varphi_n^2 \sum_i\left| \ee_k^* \frac A{\sqrt{n}} \ee_i\right|^2 \sum_j \left| \ee_j^* Q\frac {\ones}{\sqrt{n}}\right|^2 +  2\varphi_n^2 \sum_i\left| \ee_k^* \ee_i\right|^2 \sum_j \left| \ee_j^* \frac A{\sqrt{n}} Q\frac {\ones}{\sqrt{n}}\right|^2 \ ,\\
&\le& 2\varphi_n^2 \left( \ee_k^* \frac A{\sqrt{n}} \frac {A^*}{\sqrt{n}} \ee_k\right)\left( 
\frac {\ones^*}{\sqrt{n}}Q^*Q\frac {\ones}{\sqrt{n}}
\right)+2\varphi_n^2 \left( 
\frac {\ones^*}{\sqrt{n}}Q^*\frac {A^*A}n Q\frac {\ones}{\sqrt{n}}
\right)\quad \le \quad 2^2\times 3^4\, .
\end{eqnarray*}
The term $T_{3,ij}$ can be handled similarly and one can prove
$$
\sum_{ij} \left| T_{3,ij}\right|^2\quad \le \quad 3^8\, .
$$
Gathering all these estimates, we finally obtain the desired bound:
$$
\sqrt{\sum_{ij} \left| \pij \Rtilde \right|^2} \ \le \ 
\sqrt{3\sum_{ij} \left|T_{1,ij}\right|^2 +  3\sum_{ij} \left|T_{2,ij}\right|^2 + 3\sum_{ij}\left|T_{3,ij}\right|^2} 
\ \le\  K\, ,
$$
where $K$ neither depends on $k$ nor on $n$.

Having proved a local estimate over $\|\nabla \widetilde R_k(A)\|$ for each matrix $A$ such that $\HH(A)$ has simple spectrum, we now establish the Lipschitz estimate \eqref{eq:lipschitz} for two such matrices $A,B$.

Let $A,B$ such that $\HH(A)$ and $\HH(B)$ have simple spectrum and consider $A_t=(1-t)A+tB$ for $t\in [0,1]$. 
Notice first that the continuity of the eigenvalues implies that there exists $\delta>0$ sufficiently small such that $\HH(A_t)$ has a simple spectrum for $t\le \delta$ and $t\ge 1-\delta$. To go beyond $[0,\delta)\cup(1-\delta, 1]$ and prove that $\HH(A_t)$ has simple spectrum for the entire interval $[0,1]$ except maybe for a finite number of points, we rely on the argument in Kato \cite[Chapter 2.1]{kato2013perturbation} which states that apart from a finite number of $t_\ell$'s:  
$$
t_0=0<t_1<\cdots<t_L<t_{L+1}=1\ ,
$$ 
the number of eigenvalues of $\HH(A_t)$ remains constant for $t\in [0,1]$ and $t\neq t_\ell,\ell\in[L]$. Since $\HH(A_t)$ has simple spectrum for 
$t\in [0,\delta)\cup(1-\delta, 1]$, it has simple spectrum for all $t\notin \{t_\ell,\ell\in[L]\}$. 

We can now proceed:
\begin{eqnarray*}
\left| \widetilde R_k(A_{t_1}) - \widetilde R_k(A)\right| & = & \left| \lim_{\tau \nearrow t_1} \int_0^{\tau} \frac{d}{dt} \widetilde R_k(A_t) \, dt\right|
\quad =\quad  \left| \lim_{\tau \nearrow t_1} \int_0^{\tau} \nabla \widetilde R_k(A_t)\circ \frac{d}{dt} A_t \, dt\right|\ ,\\
&\le &  \lim_{\tau \nearrow t_1} \int_0^{\tau} \| \nabla \widetilde R_k(A_t)\| \times  \| B-A\|  \, dt\quad \le \quad K\, t_1\, \| B-A\|\ .
\end{eqnarray*}
By iterating this process, we obtain
$$
\left| \widetilde R_k(B) - \widetilde R_k(A)\right| \le \sum_{\ell=1}^{L+1} \left| \widetilde R_k(A_{t_\ell}) - \widetilde R_k(A_{t_{\ell-1}})\right|
\le \sum_{\ell=1}^{L+1} K( t_\ell - t_{\ell-1}) \|B-A\| = K\| B-A\|\, ,
$$
hence the Lipschitz property along the segment $[A,B]$ for $\HH(A)$ and $\HH(B)$ with simple spectrum. 

The general property follows by density of such matrices in the set of $n\times n$ matrices and by continuity of $A\mapsto \widetilde R_k(A)$. Let $A,B$ be given and $A_{\varepsilon}\to A$ and $B_{\varepsilon}\to B$ be such that $\HH(A_{\varepsilon})$ and $\HH(B_{\varepsilon})$ have simple spectrum then:
$$
\left| \widetilde R_k(B) - \widetilde R_k(A)\right|\le 
\left| \widetilde R_k(B_{\varepsilon}) - \widetilde R_k(B)\right| + K\| B_{\varepsilon}-A_{\varepsilon}\|
 + \left| \widetilde R_k(A_{\varepsilon}) - \widetilde R_k(A)\right|\xrightarrow[\varepsilon\to 0]{} K\| B-A\|\, .
$$
Proof of Lemma \ref{lem:LIP} is completed.
\end{proof}
We now use concentration arguments to obtain a bound on $\e \max_{k\in [n]} (\Rtilde - \e \Rtilde)$.
\begin{prop}\label{prop:CONC} Let $K$ be the constant obtained in Lemma \ref{lem:LIP}, then 
$$
\e \max_{k\in [n] } (\Rtilde - \e \Rtilde)\ \le \ K\sqrt{2\log n}\, .
$$
\end{prop}
\begin{proof} By applying Tsirelson-Ibragimov-Sudakov inequality \cite[Theorem 5.5]{book-boucheron-et-al-2013} to $\Rtilde(A)$ with the Lipschitz estimate obtained in Lemma \ref{lem:LIP}, we obtain the following exponential estimate:
$$
\e e^{\lambda (\Rtilde(A) - \e \Rtilde(A))} \ \le \ e^{\frac{\lambda^2 K^2}2}
$$ 
for all $\lambda\in \R$. We can now estimate the expectation of the maximum (we drop the dependence in $A$).
$$
\exp\left(\lambda \E \max_{k\in [n]} (\Rtilde - \e \Rtilde)\right)\ \le\  \E \exp\left(\lambda \max_{k\in [n]} (\Rtilde - \e \Rtilde)\right)\
\le\ \sum_{k=1}^n \E e^{\lambda (\Rtilde - \e \Rtilde)} \ \le \ n e^{\frac{\lambda^2 K^2}2}\, .
$$
Hence for $\lambda >0$
$$
\E \max_{k\in [n]} (\Rtilde - \e \Rtilde)\ \le\ \frac{\log n}{\lambda} + \frac{\lambda K^2}2\ =:\ \Phi(\lambda)\, .
$$
Optimizing in $\lambda$, we obtain $\lambda^*=\frac{\sqrt{2\log n}}{K}$ and $\Phi(\lambda^*)=K\sqrt{2\log n}$, which is the desired estimate.
\end{proof}

%\begin{prop}\label{prop:MEAN} The following convergences hold true: for all $k\in \mathbb{N}$,
%$$
%\e R_k(A) \xrightarrow[n\to\infty]{} 0 \qquad \textrm{and}\qquad \e \Rtilde(A) \xrightarrow[n\to\infty]{} 0 \ .
%$$
%\textcolor{blue}{[pas sûr qu'on ait besoin de cette proposition]}
%\end{prop}

\begin{prop}\label{prop:TIGHT}
The following estimate holds\footnote{Notice that the proof does not rely on the fact that the entries are Gaussian. In particular, we did not use the integration by part formula $\E Xf(X)=\E f'(X)$, only valid for $X\sim {\mathcal N}(0,1)$.}:
$$
\E \widetilde R_k(A_n) \ =\ {\mathcal O}\left( \alpha_n \right)\,,
$$
uniformly in $k\in [n]$.
\end{prop}

\begin{proof} Given an almost surely differentiable function $\Psi:\R\to \R$, we shall use the following Taylor expansion:
$$
\Psi(X) = \Psi(0) +X\int_0^1 \Psi'(sX)\, ds\, .
$$
We have
\begin{eqnarray*}
\E \widetilde R_k(A_n) &=& \frac 1n \sum_{i,\ell} \E\, \varphi_n A_{ki} \left[AQ\right]_{i\ell} \ =\ \frac {\alpha}{\sqrt{n}} \sum_{i,\ell} \E\, \varphi_n A_{ki} \left[\left( -I+\frac{A}{\alpha\sqrt{n}} +I\right)Q\right]_{i\ell}\ ,\\
&=&  \frac \alpha{\sqrt{n}} \sum_{i,\ell} \E A_{ki} \underbrace{\varphi_n \left( -\delta_{i\ell} + Q_{i\ell}\right)}_{:=\Psi_{i\ell}(A_{ki})} 
\ =\ \frac \alpha{\sqrt{n}} \sum_{i,\ell} \E A_{ki} \left( \Psi_{i\ell}(0) + A_{ki} \int_0^1 \Psi_{i\ell}(sA_{ki})\, ds\right)\ .\\ 
\end{eqnarray*}
Notice that $\Psi_{i\ell}(0)$ does not depend on $A_{ki}$ anymore, hence is independent from this random variable. In particular $\E A_{ki} \Psi_{i\ell}(0)=0$.
We denote by $\underline{F}$ a function $F$ evaluated at $sA_{ki}$, i.e. $\underline{F}=F(sA_{ki})$. We have
\begin{eqnarray*}
\E \widetilde R_k(A_n) &=& -\frac \alpha{\sqrt{n}} \sum_i \int_0^1 ds\, \E A_{ki}^2 \underline{\pik \varphi_n}  \\
&&\qquad +\frac \alpha{\sqrt{n}} \sum_{i,\ell} \int_0^1 ds\, \E A_{ki}^2 \left( \underline{\pik \varphi_n}\right) \underline{Q_{i\ell}} 
+\frac \alpha{\sqrt{n}} \sum_{i,\ell} \int_0^1 ds\, \E A_{ki}^2 \underline{\varphi_n} \left( \underline{\pik  Q_{i\ell}}\right) \, ,\\
&=:& T_1+T_2+T_3\, .
\end{eqnarray*}

Recall that $\pij \varphi_n= \varphi'(s_n) \pij s_n$, where $\pij s_n$ has been computed in \eqref{eq:sn-derivative}. Denote by $u_k=\boldsymbol{u}^* \ee_k$ and $v_i =\ee_i^* \boldsymbol{v}$ and recall that $\| \boldsymbol{v}\|=1$ and $|u_k|,|v_i|\le 1$. We have
\begin{eqnarray*}
|T_1|&=& \left| \frac \alpha{n}  \sum_{i}\int_0^1 ds\, \E\, A_{ki}^2 \underline{\varphi'(s_n)}\,\underline{u_k}\,\underline{ v_i} \right|\ \le\ \alpha\, \E\, A_{11}^2 \, . 
\end{eqnarray*}
Hence $T_1={\mathcal O} \left(\alpha\right)$. 
Now
\begin{eqnarray*}
|T_2|&=& \left| \frac \alpha{n}  \sum_{i,\ell}\int_0^1 ds\, \E\, A_{ki}^2 \underline{\varphi'(s_n)}\underline{u_k}\,\underline{ v_i} \, \underline{Q_{i\ell}}\right|
\ \le\ \alpha \int_0^1ds\, \E \,\underline{\varphi'(s_n)}\,\left| \boldsymbol{a}_k \underline{Q} (\ones/\sqrt{n})\right|\  \textrm{where}\ \boldsymbol{a}_k=\left( \frac{\underline{u_k} A_{ki}^2 \underline{v_i} }{\sqrt{n}}\right)_{i\in [n]}\, ,\\
&\le& \alpha \int_0^1 ds\, \E \| \boldsymbol{a}_k\|\ \le \ \alpha \int_0^1 ds\, \left( \E \| \boldsymbol{a}_k\|^2\right)^{1/2}\ \le \ \alpha \left(\E A_{11}^4\right)^{1/2}\, .
\end{eqnarray*}
Hence $T_2={\mathcal O} (\alpha)$. Finally
\begin{eqnarray*}
|T_3|&=& \left| 
\frac 1n \sum_{i,\ell}\int_0^1 ds\, \E\, A_{ki}^2 \underline{\varphi_n} \underline{Q_{ik}} \underline{Q_{i\ell}}\right|
\ =\ \left| \int_0^1ds\, \E \, \varphi_n (\boldsymbol{b}_k \underline{Q} \ones/\sqrt{n}) \right|\ 
\textrm{where}\ \boldsymbol{b}_k = \left( \frac{A_{ki}^2 Q_{ik}}{\sqrt{n}} \right)_{i\in[n]}\ ,\\
&\le& \int_0^1 ds\, \E \varphi_n \| \boldsymbol{b}_k\| \ \le \ 3\left( \E A_{11}^4\right)^{1/2}\, .
\end{eqnarray*}
Hence $T_3={\mathcal O} (1)$.
 
We have finally proven that $\E \widetilde R_k(A_n) ={\mathcal O} \left(\alpha \right)$ uniformly in $k$, which concludes the proof of the lemma.
\end{proof}

We are now in position to prove Lemma \ref{lemma:main}.

\subsection{Proof of Lemma \ref{lemma:main}}\label{sec:proof-lemma-main}
We first establish the convergence for $\max_{k\in [n]}\Rtilde(A) - \widetilde R_1(A)$. Notice that the r.v. $\max_{k\in [n]} \Rtilde(A) - \widetilde R_1(A)$ is nonnegative hence by Markov inequality,
\begin{eqnarray*}
\mathbb{P} \left\{ \frac{\max_{k\in [n]} \Rtilde(A) - \widetilde R_1(A)}{\alpha \sqrt{2\log n}} \ge \varepsilon \right\} 
&\le&\frac{\e \left(\max_{k\in [n]} \Rtilde(A) - \widetilde R_1(A)\right)}{\varepsilon \alpha \sqrt{2\log n}}
\ ,\\
&=& \frac{\e \left(\max_{k\in [n]} \left(\Rtilde(A) -\E\Rtilde(A) + \E\Rtilde(A)\right) - \widetilde R_1(A)\right)}{\varepsilon \alpha \sqrt{2\log n}}
\ ,\\
&\le& \frac{\e \left(\max_{k\in [n]} \left(\Rtilde(A) -\E\Rtilde(A) \right) + \max_{k\in[n]} \E\Rtilde(A) - \widetilde R_1(A)\right)}{\varepsilon \alpha \sqrt{2\log n}}
\ ,\\
&\le& \frac{\e \left(\max_{k\in [n]} \left(\Rtilde(A) -\E\Rtilde(A) \right) \right)}{\varepsilon \alpha \sqrt{2\log n}} + \frac{ \max_{k\in[n]} \E\Rtilde(A) - \e \widetilde R_1(A)}{\varepsilon \alpha \sqrt{2\log n}}\, .
\end{eqnarray*}
Now since the random variables $\widetilde R_1(A), \dots, \widetilde R_n(A)$ are exchangeable, $\max_{k\in[n]} \E\Rtilde(A) = \e \widetilde R_1(A)$ and
$$
\mathbb{P} \left\{ \frac{\max_{k\in [n]} \Rtilde(A) - \widetilde R_1(A)}{\alpha \sqrt{2\log n}} \ge \varepsilon \right\} 
\,\, \le\,\, \frac{\e \max_{k\in [n]} \left(\Rtilde(A) - \e \widetilde R_k(A)\right)}{\varepsilon \alpha \sqrt{2\log n}}
\,\, \le \,\, \frac K{\varepsilon \alpha}\ \xrightarrow[n\to\infty]{} \ 0
$$
by Proposition \ref{prop:CONC}. This implies that 
\begin{equation}\label{eq:probab2}
\frac{\max_{k\in [n]} \Rtilde(A) - \widetilde R_1(A)}{\alpha \sqrt{2\log n}} \xrightarrow[n\to\infty]{\mathcal P} 0\ .
\end{equation}
We now prove that 
\begin{equation}\label{eq:probab}
\frac{\widetilde R_1(A)}{\alpha \sqrt{2\log n}}\xrightarrow[n\to\infty]{\mathcal P}0\, .
\end{equation} By Proposition \ref{prop:TIGHT}, $\E \widetilde R_1(A)={\mathcal O}(\alpha)$ hence $\E \widetilde R_1(A)/(\alpha\sqrt{2\log(n)}) \to 0$. Applying Poincaré's inequality to the Lipschitz functional $A\mapsto \widetilde R_1(A)$ (cf. Lemma \ref{lem:LIP}), we can bound $\widetilde R_1(A)$'s variance by $L^2$ and obtain 
$$
\mathbb{P}\left(\left| \frac{\widetilde R_1(A) - \E \widetilde R_1(A)}{\alpha \sqrt{2\log n}}\right| >\delta \right) \ \le\  \frac{\var (\widetilde R_1(A))}{2 \delta^2 \alpha^2 \log n} 
\ \le \ \frac{L^2} {2 \delta^2 \alpha^2 \log n}\quad\xrightarrow[n\to\infty]{}\quad 0\, .
$$
This yields \eqref{eq:probab}. Combining \eqref{eq:probab2} and \eqref{eq:probab} finally yields:
$$
\frac{\max_{k\in [n]} \widetilde R_k(A)}{\alpha \sqrt{2\log n}}\  \xrightarrow[n\to\infty]{\mathcal P} \ 0\ .
$$
In order to obtain the result for the untilded quantities, we write
\begin{eqnarray*}
\mathbb{P} \left\{ \left| \frac{\max_{k} R_k(A) }{\alpha \sqrt{2\log n}} \right| >\varepsilon\right\} &\le & 
\mathbb{P} \left\{ \left| \frac{\max_{k} R_k(A) -\max_{k} \Rtilde(A)}{\alpha \sqrt{2\log n}} \right| >\varepsilon/2\right\} 
+ \mathbb{P} \left\{ \left| \frac{\max_{k} \Rtilde(A) }{\alpha \sqrt{2\log n}} \right| >\varepsilon/2\right\} \ ,\\
&=& \mathbb{P}\{ \varphi_n <1\} +  \mathbb{P} \left\{ \left| \frac{\max_{k} \Rtilde(A) }{\alpha \sqrt{2\log n}} \right| >\varepsilon/2\right\} 
\xrightarrow[n\to\infty]{} 0\, .
\end{eqnarray*}
 One proves the second assertion similarly, which concludes the proof of Lemma \ref{lemma:main}.

\section{Stability: proof of Theorem \ref{th:stability}}\label{section:proof-stability}

In order to study the stability of large Lotka-Volterra systems, we are led to study the matrix 
$$
\J(\x_n) = \diag(\x_n)\left(-I_n + \frac {A_n}{\alpha_n\sqrt{n}}\right)\ .
$$
%\begin{prop}
%Let $\lplus=\limsup_{n\to \infty}\frac{\sqrt{2\log n}}{\alpha_n}$ and assume that $\lplus<1$; denote by ${\mathcal S}_n$ the spectrum of $\J(\x_n)$. Then for every $\lambda \in {\mathcal S}_n$, 
%$$
%\min_{k\in [n]} \left| \lambda + x_k\right| \xrightarrow[n\to\infty]{\mathcal P} 0\, .
%$$
%In particular,
%\begin{equation}\label{eq:stability}
%\mathrm{Re}\, \lambda \le -(1-\lplus) + o_P(1)\, .
%\end{equation}
%\end{prop}
%
%\begin{proof}
We first establish the following estimates
\begin{equation}\label{eq:estimates-stability}
\begin{cases}
\min_{k\in[n]} x_k \ge 1-\lplus -o_P(1)\, ,\\
\max_{k\in [n]} x_k \le 1+\lplus +o_P(1)\, .
\end{cases}
\end{equation}
The first estimate immediatly follows from \eqref{eq:useful} together with Lemma \ref{lemma:main}. From $x_k$'s decomposition \eqref{eq:decomp} we have
$$
\max_{k\in[n]} x_k \le  1+\frac {M_n}{\alpha_n} +\frac{\max_{k\in[n]} R_k}{\alpha_n^2} = 1+\frac{\alpha^*_n}{\alpha_n}\left( \frac{M_n-\beta^*_n}{\alpha^*_n} +\frac{\beta^*_n}{\alpha^*_n} 
+ \frac{\max_{k\in[n]} R_k}{\alpha^*_n\alpha_n}\right) \le 1+\lplus +o_P(1)\, ,
$$
where the last inequality follows from Lemma \ref{lemma:main} and the fact that $\left(\alpha^*_n\right)^{-1}(M_n-\beta^*_n)\xrightarrow[]{\mathcal P} 0$.

We now compare the spectra of matrices ${\mathcal D}(\x_n)= -\diag(\x_n)$ and $\J(\x_n)$ by relying on Bauer and Fike's theorem \cite[Theorem 6.3.2]{book-horn-johnson}:
for every $\lambda\in {\mathcal S}_n$, there exists a component $x_k$ of vector $\x_n$ such that 
\begin{eqnarray*}
|\lambda + x_k| & \le & \left\| \diag(\x_n)\frac{A_n}{\alpha_n\sqrt{n}} \right\| \ \le \ \frac 1{\alpha_n} \left\| \diag(\x_n) \right\| \, \left\| \frac {A_n}{\sqrt{n}}\right\|
\ \stackrel{(a)}{\le} \ \frac 1{\alpha_n} \left( 1+\lplus +o_P(1) \right) \left( 2+o_P(1) \right)\\
&=& o_P(1)\, .
\end{eqnarray*}
where $(a)$ follows from the second estimate in \eqref{eq:estimates-stability} and from the spectral norm estimate. Notice that the majorization above is uniform for $\lambda\in {\mathcal S}_n$. The first part of the theorem is proved. Finally,
$$
\mathrm{Re} (\lambda) +x_k \le |\lambda +x_k| =o_P(1)\qquad \Rightarrow \qquad \mathrm{Re} (\lambda) \le - \min_{k\in [n]} x_k + o_P(1)\, .
$$
The estimate \eqref{eq:stability} finally follows from the first estimate in \eqref{eq:estimates-stability}.
%\end{proof}

\section{Heuristics at critical scaling, non-homogeneous systems and non-gaussian entries}\label{section:various}
\subsection{A heuristics at the critical scaling}
We provide here a heuristics to compute the probability that a solution $\x_n$ is feasible at critical scaling $\alpha_n^*=\sqrt{2\log n}$.
\begin{heuristics} The probability that a solution is feasible at the critical scaling $\alpha_n^*$ is asymptotically given by
\begin{equation}\label{eq:critical}
\mathbb{P}( x_k>0,\ k\in [n])\ \approx\ 
 1- \sqrt{\frac {e}{4\pi\log n}}+ \frac{e}{8\pi\log n} =: H_1(n)\, .
 \end{equation}
\end{heuristics}

In Figure \ref{fig:critical}, we compare the heuristics with results from simulations. 

\begin{figure}[ht] 
\centering
  \includegraphics[width=\linewidth]{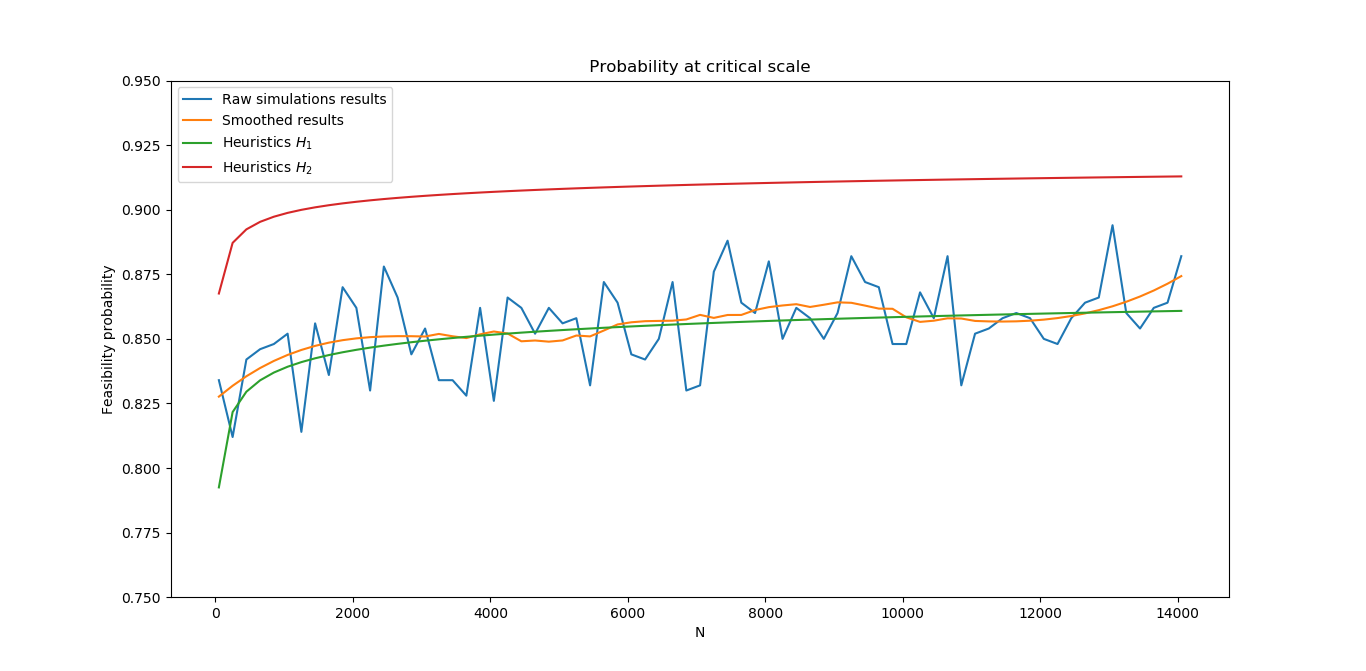}
 \caption{Probability at critical scaling. The blue curve corresponds to the proportion of feasible solutions at critical scaling $\alpha_N^*$ obtained for 500 simulations (for $N$ ranging from 50 to 14050 with a 200-increment) - notice the strong standard deviation. The yellow curve is a smoothed version of the previous curve, obtained by applying a Savitsky-Golay filter. The green curve represents the heuristics $H_1$ defined in \eqref{eq:critical}. The red curve represents the heuristics 
 $H_2$ introduced in Remark \ref{rem:H2}. Notice the substantial discrepancy between $H_1$ and $H_2$.}
\label{fig:critical}
\end{figure}

\begin{proof}[Arguments]
Consider 
$$
x_k = 1+\ee_k^* \frac{A_n}{\alpha_n^* \sqrt{n}} \ones_n + \frac{R_k}{(\alpha_n^*)^2} 
= 1+\frac{Z_k}{\alpha_n^*} + \frac{R_k}{(\alpha_n^*)^2}
= 1+\frac1{\alpha_n^*}\left( Z_k  + \frac{R_k}{\alpha_n^*}\right)\ .
$$
Following Geman and Hwang \cite[Lemma A.1]{geman-1982}, one could prove that $Z_k$ and $R_k$ are asymptotically independent centered 
Gaussian random variables, each with variance one. We thus approximate the quantity $Z_k  + \frac{R_k}{\alpha_n^*}$ by a Gaussian random variable 
with distribution ${\mathcal N}\left(0, 1+\frac 1{(\alpha_n^*)^{2}}\right)$ and set 
\begin{eqnarray*}
x_k &\approx & 1+ \left( \frac{1}{\alpha_n^*}\sqrt{ 1 +\frac 1{(\alpha_n^*)^2}} \right)U_k 
\end{eqnarray*}
where the $U_k$'s are i.i.d. ${\mathcal N}(0,1)$. 
Denote by $\check M^U_n=\min_{k\in [n]} U_k$ then
\begin{eqnarray*}
\PP(x_k>0\, , \ k\in [n]) &\approx& \PP \left( 1 + \left( \frac{1}{\alpha_n^*}\sqrt{ 1 +\frac 1{(\alpha_n^*)^2}} \right)\check M^U_n  >0\right) \ .
\end{eqnarray*}
Recall that standard extreme value convergence results for Gaussian i.i.d. random variables yield 
\begin{equation}\label{eq:EVTzero}
\PP\left\{ \alpha^*_n \left( -\check{M}^U_n - \beta^*_n\right)< x\right\} 
= \PP \left\{ \alpha^*_n ( \check{M}^U_n +\beta^*_n ) > -x\right\} 
\quad \xrightarrow[n\to\infty]{}\quad  G(x)=e^{-e^{-x}}\, ,
\end{equation}
where $\beta_n^*$ is defined in \eqref{def:beta}. Denote by $\Theta(\alpha) = \sqrt{ 1 +\alpha^{-2}}$ then 
$$
\PP \left( 1 +\Theta(\alpha^*_n)\frac{\check{M}^U_n}{\alpha_n^*} >0\right) =\PP \left( \check{M}^U_n> - \frac{\alpha_n^*}{\Theta(\alpha_n^*)}\right)
= \PP \left( \alpha^*_n(\check{M}_n +\beta^*_n) > -\frac{(\alpha^*_n)^2}{\Theta(\alpha_n^*)} +\alpha^*_n \beta^*_n\right)\ .
$$
Notice that 
$$
-\frac{(\alpha^*_n)^2}{\Theta(\alpha_n^*)} +\alpha^*_n \beta^*_n = \frac 12 - \frac 12  \log(4\pi\log n) +{\mathcal O}\left( \frac 1{(\alpha^*_n)^2}\right)
= \frac 12 +\log \frac 1{\sqrt{2\pi} \alpha_n^*} + {\mathcal O}\left( \frac 1{(\alpha^*_n)^2}\right)\ .
$$
Hence
\begin{eqnarray}
\PP \left( 1 +\Theta(\alpha^*_n)\frac{\check{M}^U_n}{\alpha_n^*} >0\right) 
&=& \PP \left( \alpha^*_n ( \check{M}_n +\beta^*_n) >  \frac 12 +\log \frac 1{\sqrt{2\pi} \alpha_n^*} + {\mathcal O}\left( \frac 1{(\alpha^*_n)^2}\right)\right)\ ,\nonumber\\
&\stackrel{(a)}\approx & e^{-\exp\left(\frac 12 +\log \frac 1{\sqrt{2\pi} \alpha_n^*} + {\mathcal O}\left( \frac 1{(\alpha^*_n)^2}\right)\right)} 
\ =\ e^{-\sqrt{\frac{e}{2\pi}}\frac 1{\alpha_n^*}\left( 1+{\mathcal O} ( (\alpha_n^*)^{-2})\right)}\ ,\nonumber\\
&=& 1-\sqrt{\frac e{2\pi}}\frac 1{\alpha_n^*} +\frac 12 \frac e{2\pi}\frac 1{(\alpha_n^*)^2} +{\mathcal O}\left( \frac{1}{(\alpha_n^*)^{3}} \right)\ .
\label{eq:approximate}
\end{eqnarray}
We finally end up with the announced approximation
$$
\PP(x_k>0\, , \ k\in [n])\quad  \approx\quad  H_1(n):=1 - \sqrt{ \frac e{4\pi \log n}} + \frac e{8\pi \log n}\ . 
$$

\begin{rem} \label{rem:H2} A rougher approximation would have been to set $x_k\approx 1+ \frac{Z_k}{\alpha_n^*}$ with $Z_k\sim{\mathcal N}(0,1)$ and to drop the next term $\frac{R_k}{(\alpha_n^*)^2}$ in the heuristics but this would have resulted in the following approximation
$$
\mathbb{P}( x_k>0,\ k\in [n])\quad \approx\quad 1 - \left( 4\pi \log(n)\right)^{-1/2}+\left( 8\pi \log(n)\right)^{-1} =:H_2(n)\, ,
$$ 
which is worst than $H_1(n)$, as illustrated in Figure \ref{fig:critical}.
\end{rem}

Approximation $(a)$ in \eqref{eq:approximate} may look doubtful, especially because the convergence \eqref{eq:EVTzero} is used for growing $x\sim\log(\log n)$. Since it is well-known that convergence in distribution might not capture the convergence of the tails, one may want to switch to the regime of large deviations. We rely on computations made by Vivo \cite{vivo2015large} to confirm that the approximation $(a)$ is legitimate. 

The following large deviations estimate is provided in \cite[Eq. (52)]{vivo2015large}:
$$
 \lim_{n\to\infty} \frac{\log \PP\left\{ M^U_n \ge \xi \beta^*_n\right\} }{ \log n} = - (\xi^2 - 1)\qquad \textrm{for}\quad \xi\ge 1\, ,
$$
which yields the approximation 
\begin{equation}\label{eq:LDP}
\PP\left\{ M^U_n \ge \xi \beta^*_n\right\}  \approx e^{- (\log n)(\xi^2-1)}\, .
\end{equation}
On the other hand, by classical extreme value theory, 
\begin{equation}\label{eq:EVT}
\PP\left\{ M^U_n > (\alpha^*_n)^{-1} x + \beta^*_n\right\} \approx 1 - e^{-e^{-x}}\, .
\end{equation} 
Now, in order to extend the validity of \eqref{eq:EVT} for $x\gg 1$, we consider simultaneously the approximation \eqref{eq:LDP} for $\xi \sim 1$ and \eqref{eq:EVT} for $x\gg 1$, that is
\begin{eqnarray*}
\PP\left\{ M^U_n \ge \xi \beta^*_n\right\}  &\approx& e^{- (\log n)(\xi^2-1)}\approx e^{- 2(\log n)(\xi-1)}\quad \textrm{for}\quad \xi\sim 1\\
\PP\left\{ M^U_n > (\alpha^*_n)^{-1} x + \beta^*_n\right\} &\approx& 1 - e^{-e^{-x}} \approx e^{-x} \quad \textrm{for}\quad x\gg 1
\end{eqnarray*} 
Equating both exponentials yields 
$$
x=2\log n (\xi -1) \quad \Rightarrow \quad \xi = 1+\frac x{2\log n}\, .
$$
This gives us the following rule of thumb: one may apply \eqref{eq:EVT} if $1\ll x\ll \log n$. This condition is fulfilled for $x\sim \log (\log n)$.

\end{proof}

\subsection{Positivity for a non-homogeneous linear system} The results developed so far for the system \eqref{eq:equilibrium} extend to a non-homogeneous (NH) linear system where $\ones_n$ is replaced by a deterministic $n\times 1$ vector $\bs{r}_n$ with slight modifications. In particular, we identify a regime where feasibility and stability occur simultaneously.

Denote by $\br_n=(r_k)$ a $n\times 1$ deterministic vector with positive components and consider the linear system
\begin{equation}\label{eq:non-homo}
\x_n=\br_n +\frac{1}{\alpha_n\sqrt{n}} A_n \x_n\, .
\end{equation}
Introduce the notations 
$$
r_{\min}(n)=\min_{ k\in[n]} r_k\,,\quad  r_{\max}(n) =\max_{k\in[n]} r_k\quad \textrm{and}\quad \sigma_{\br}(n) =\| \br/\sqrt{n}\|=\sqrt{n^{-1} \sum_{k\in[n]} r_k^2}\ .
$$ 
Assume that there exist $\rho_{\min}, \rho_{\max}$ independent from $n$ such that eventually
$$
0\quad <\quad \rho_{\min} \quad \le\quad  r_{\min}(n) \quad \le \quad \sigma_{\br}(n)\quad \le\quad r_{\max}(n) \quad \le\quad \rho_{\max} <\infty\, .
$$
Then 
\begin{theo}[Feasibility - NH case]\label{th:non-homo}
Let $\alpha_n \xrightarrow[n\to\infty]{} \infty$  and denote by $\alpha_n^*=\sqrt{2\log n}$. Let $\x_n=(x_k)_{k\in [n]}$ be the solution of \eqref{eq:non-homo}. 
\begin{enumerate}
\item If there exists $\varepsilon>0$ such that eventually $\alpha_n\le (1-\varepsilon) \frac{\alpha_n^* \sigma_{\br}(n)}{r_{\max}(n)}$ then 
$
\mathbb{P}\left\{ \min_{k\in [n]} x_k>0\right\} \xrightarrow[n\to\infty]{} 0\, .
$

\item If there exists $\varepsilon>0$ such that eventually $\alpha_n\ge (1+\varepsilon)\frac{\alpha_n^* \sigma_{\br}(n)}{r_{\min}(n)}$ then
$
\mathbb{P}\left\{ \min_{k\in [n]} x_k>0\right\} \xrightarrow[n\to\infty]{} 1\, .
$
\end{enumerate}
\end{theo}
\begin{rem} Contrary to the homogeneous system where there is a sharp transition at $\alpha^*_n=\sqrt{2\log(n)}$, the situation is not as clean-cut here and there is a buffer zone 
$$
\alpha_n \in \left[ \frac{ \sigma_{\br}(n)}{r_{\max}(n)}\sqrt{2\log(n)} \, , \, \frac{\sigma_{\br}(n)}{r_{\min}(n)} \sqrt{2\log(n)}\right]
$$ 
in which the study of the feasibility is not clear. This buffer zone is illustrated in Figure \ref{fig:NH}. 
\end{rem}

In Figure \ref{fig:NH}, we illustrate the transition toward feasibility for a non-homogeneous system \eqref{eq:non-homo} in the case where 
deterministic vector $\br_N$ is equally distributed over $[1,3]$, i.e.
\begin{equation}\label{eq:def-rn}
\br_N(i)= 1+\frac{2i}N\, , \quad 1\le i\le N\, .
\end{equation}
We introduce the quantities
\begin{equation}\label{eq:threshold}
t_1=\lim_N \frac{\sqrt{2} \sigma_{\br}(N)}{\br_{\max}}= \frac{\sqrt{2\int_0^1 (1+2x)^2\, dx}}3=0.98\qquad \textrm{and}\qquad 
t_2=\lim_N \frac{\sqrt{2} \sigma_{\br}(N)}{\br_{\min}}=\sqrt{2\int_0^1 (1+2x)^2\, dx}=2.94\, .
\end{equation}
As one may notice, the transition region is wider than in the homogeneous case.
\begin{figure}[ht] 
\centering
  \includegraphics[width=\linewidth]{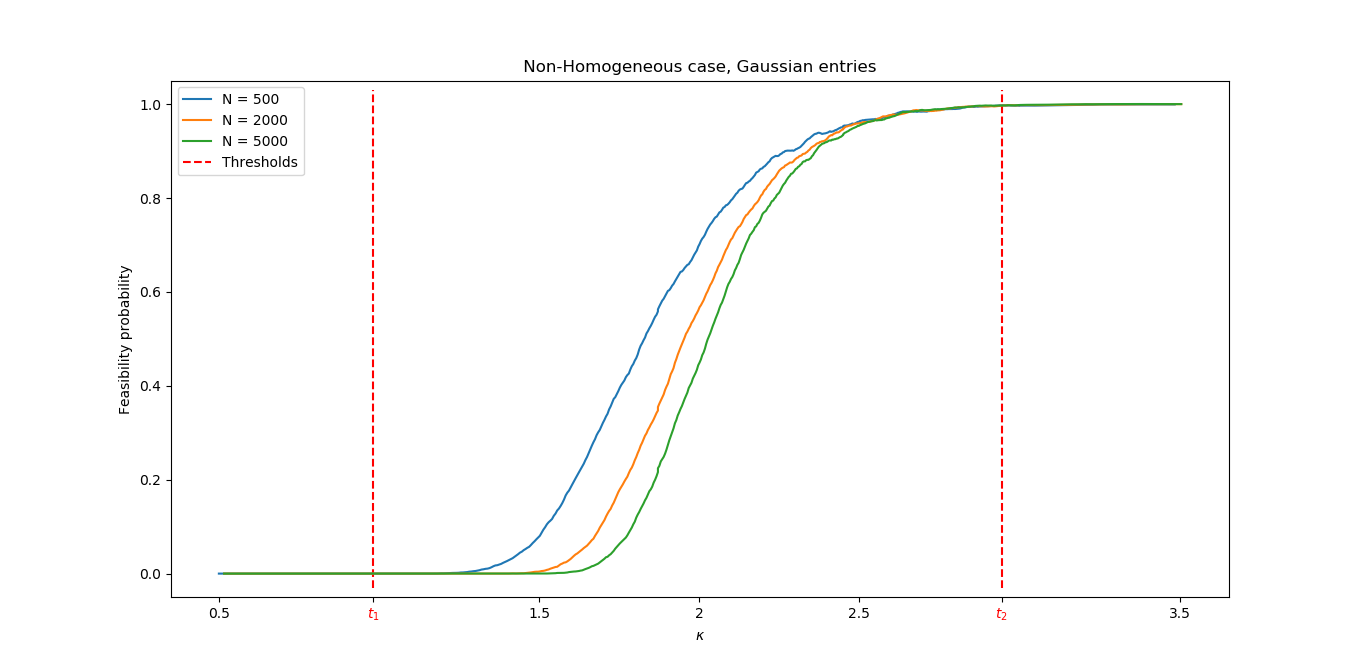}
 \caption{Transition toward feasibility for a NH system. The curves are obtained as for Figure \ref{fig:phase-transition} for $\br_N$ defined in \eqref{eq:def-rn}. The thresholds $t_1$ and $t_2$ are computed in \eqref{eq:threshold}.}
\label{fig:NH}
\end{figure}

\begin{proof}[Elements of proof] We have 
$$
x_k = \ee_k^* \, Q\, \br_n = r_k +\frac 1{\alpha} \frac{\sum_{i=1}^n r_i A_{ki}}{\sqrt{n}} + \frac 1{\alpha^2} \ee^*_k \left( \frac A{\sqrt{n}}\right) Q\, \br_n 
= r_k +\frac{\sigma_{\br}(n)}{\alpha} U_k + \frac 1{\alpha^2} R^{(\br)}_k
$$
where the $U_k$'s are i.i.d. ${\mathcal N}(0,1)$. One can check by carefully reading the proof of Lemma \ref{lemma:main} that the conclusions of the lemma apply to $\Rr_k$. In particular, one may check that Proposition \ref{prop:TIGHT} holds uniformly in $k\in [n]$ in the non-homogeneous case. Denote by $\check M=\min_{k\in[n]} U_k$, then
\begin{eqnarray*}
\min_{k\in[n]} x_k &\le & r_{\max}(n) +\frac{\sigma_{\br}(n)}{\alpha} \check M + \frac{\max_{k\in[n]}\Rr_k}{\alpha^2}\ ,\\
&\le& r_{\max}(n) +\frac{\sigma_{\br}(n) \alpha^*}{\alpha} \left( \frac{\check M +\beta^*}{\alpha^*} - \frac{\beta^*}{\alpha^*} 
+ \frac{\max_{k\in[n]}\Rr_k}{\sigma_{\br}(n) \alpha^* \alpha}\right)\quad 
=\quad r_{\max}(n) +\frac{\sigma_{\br}(n) \alpha^*}{\alpha} \left( -1 +o_P(1)\right)\ .
\end{eqnarray*}
The first statement of the theorem follows. Similarly,
\begin{eqnarray*}
\min_{k\in[n]} x_k &\ge & r_{\min}(n) +\frac{\sigma_{\br}(n)}{\alpha} \check M + \frac{\min_{k\in[n]}\Rr_k}{\alpha^2}\, , \\
&\ge& r_{\min}(n) +\frac{\sigma_{\br}(n) \alpha^*}{\alpha} \left( \frac{\check M +\beta^*}{\alpha^*} - \frac{\beta^*}{\alpha^*} 
+ \frac{\min_{k\in[n]}\Rr_k}{\sigma_{\br}(n) \alpha^* \alpha}\right)\quad 
=\quad r_{\min}(n) +\frac{\sigma_{\br}(n) \alpha^*}{\alpha} \left( -1 +o_P(1)\right)\ .
\end{eqnarray*} 
Proof of Theorem \ref{th:non-homo} is completed.
\end{proof}

A non homogeneous system \eqref{eq:non-homo} is associated to the following Lotka-Volterra system
$$
\frac{dx_k(t)}{dt} = x_k(t)\, \left( r_k - x_k(t) + \frac{1}{\alpha_n \sqrt{n}} \sum_{\ell\in[n]} A_{k\ell} x_{\ell}(t)\right)
$$
for $k\in [n]$ whose jacobian at equilibrium is still given by \eqref{eq:jacobian}. 

\begin{theo}[Stability - NH case]\label{th:stability-nonhomo} Let $\x_n=(x_k)_{k\in [n]}$ be the solution of \eqref{eq:non-homo} and assume that 
$$
\lplus\ :=\ \limsup_{n\to \infty}\frac{\alpha^*_n\,\sigma_{\br}(n)}{\alpha_n\, r_{\min}(n)}\quad <\quad 1\, .
$$ 
Denote by ${\mathcal S}_n$ the spectrum of $\J(\x_n)$. Then for every $\lambda \in {\mathcal S}_n$, 
$$
\max_{\lambda\in {\mathcal S}_n} \min_{k\in [n]} \left| \lambda + x_k\right| \xrightarrow[n\to\infty]{\mathcal P} 0\qquad \textrm{and}\qquad \max_{\lambda\in {\mathcal S}_n}\mathrm{Re}\, \lambda \le -(1-\lplus) + o_P(1)\, .
$$
\end{theo}

\subsection{Beyond the Gaussian case} The results presented so far heavily rely on the Gaussianity of the entries. A closer look at $\x_n$'s components reveals that Gaussianity plays an important role at three levels:
$$
x_k = 1 +\frac 1{\alpha}Z_k  + \frac 1{\alpha^2} R_k\quad \textrm{where} \quad  Z_k=\frac{\sum_{i\in [n]} A_{ki}}{\sqrt{n}}\ .
$$ 
\begin{enumerate}
\item Gaussian entries immediatly imply that the $Z_k$'s are independent standard Gaussian random variables, for which the study of the extrema is standard. 
\end{enumerate}
In the case where the entries are not Gaussian any more, the $Z_k$'s are no longer Gaussian but this issue can easily be circumvented since by the CLT the $Z_k$'s converge in distribution to a standard Gaussian. The extreme value study of such families of $Z_k$'s has been carried out in \cite[Propositions 2 \& 3]{anderson1997maxima}. 
\begin{enumerate}[resume]
\item The study of the extreme values of $(R_k, k\in [n])$ in this article relies on the sub-Gaussiannity of $\widetilde R_k(A)$ which is a consequence of Gaussian concentration for Lipschitz functionals.
\item Poincaré's inequality is used to prove that $\widetilde R_1(A)/(\alpha\sqrt{2\log(n)})$ goes to zero in probability, which is crucial to establish Lemma \ref{lemma:main}.
\end{enumerate}
If the distribution of the entries is strongly log-concave in the sense of \cite[Eq. (3.48)]{wainwright2019high}, then \cite[Theorem 3.16]{wainwright2019high} yields the sub-Gaussiannity of $\widetilde R_1(A)$ together with Poincaré's inequality. In particular, Theorems \ref{th:main} and \ref{th:stability} hold verbatim for entries $(A_{ij})$ i.i.d., centered with variance one and whose distribution is strongly log-concave. 

The case of bounded and/or discrete entries is not covered and remains open although the simulations (see Figure \ref{fig:NG}) indicate that a similar phase transition occurs.

\begin{figure}[ht] 
\centering
  \includegraphics[width=\linewidth]{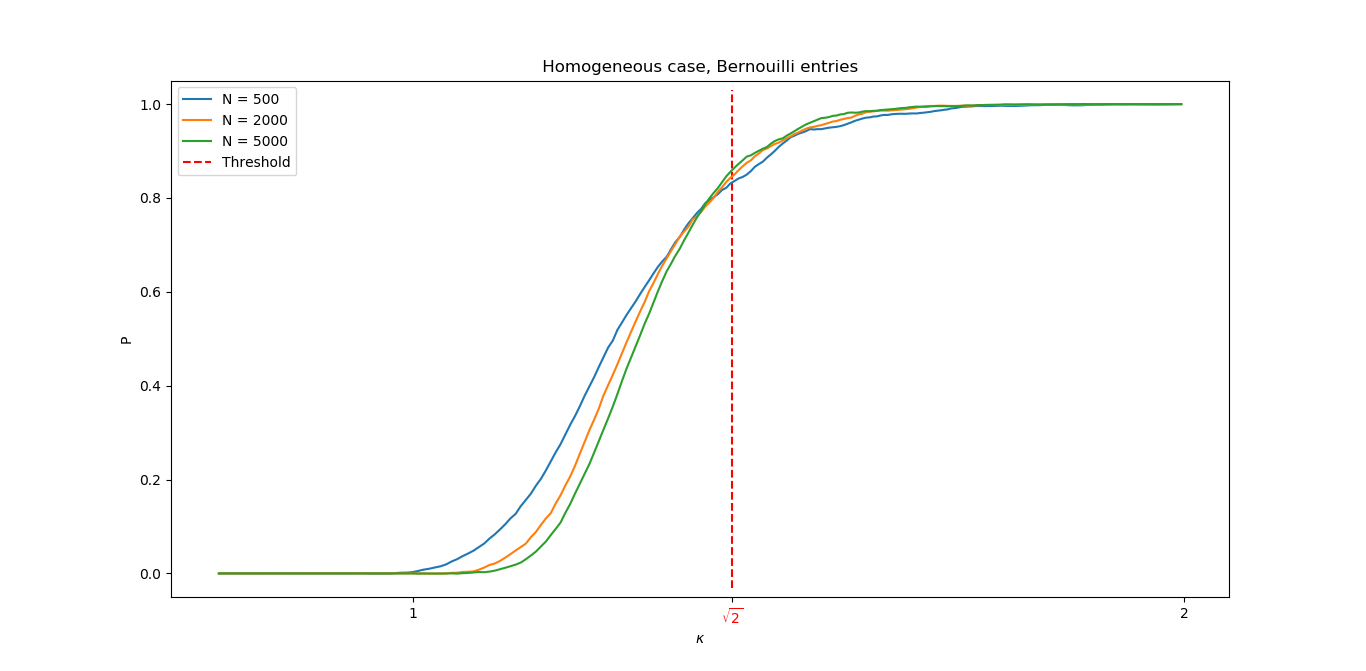}
 \caption{Transition toward feasibility for non-Gaussian entries. The curves are obtained as for Figure \ref{fig:phase-transition} in the case where the matrix's entries are Bernouilli $\pm1$.}
\label{fig:NG}
\end{figure}

\bibliographystyle{abbrv}
\bibliography{../../Bibliographies/mathematics}

\begin{thebibliography}{10}

\bibitem{Allesina:2015ux}
S.~Allesina and S.~Tang.
\newblock The stability--complexity relationship at age 40: a random matrix
  perspective.
\newblock {\em Population Ecology}, 57(1):63--75, 2015.

\bibitem{anderson1997maxima}
C.~W. Anderson, S.~G. Coles, and J.~H{\"u}sler.
\newblock Maxima of poisson-like variables and related triangular arrays.
\newblock {\em The Annals of Applied Probability}, pages 953--971, 1997.

\bibitem{book-bai-silverstein}
Z.~D. Bai and J.~W. Silverstein.
\newblock {\em Spectral analysis of large dimensional random matrices}.
\newblock Springer Series in Statistics. Springer, New York, second edition,
  2010.

\bibitem{book-boucheron-et-al-2013}
S.~Boucheron, G.~Lugosi, and P.~Massart.
\newblock {\em Concentration Inequalities: A Nonasymptotic Theory of
  Independence}.
\newblock Oxford University Press, 2013.

\bibitem{dougoud2018feasibility}
M.~Dougoud, L.~Vinckenbosch, R.~P. Rohr, L.-F. Bersier, and C.~Mazza.
\newblock The feasibility of equilibria in large ecosystems: A primary but
  neglected concept in the complexity-stability debate.
\newblock {\em PLoS computational biology}, 14(2):e1005988, 2018.

\bibitem{gardner1970connectance}
M.~R. Gardner and W.~R. Ashby.
\newblock Connectance of large dynamic (cybernetic) systems: critical values
  for stability.
\newblock {\em Nature}, 228(5273):784, 1970.

\bibitem{geman-1986}
S.~Geman.
\newblock The spectral radius of large random matrices.
\newblock {\em Ann. Probab.}, 14(4):1318--1328, 1986.

\bibitem{geman-1982}
S.~Geman and C.-R. Hwang.
\newblock A chaos hypothesis for some large systems of random equations.
\newblock {\em Z. Wahrsch. Verw. Gebiete}, 60(3):291--314, 1982.

\bibitem{gibbs2018effect}
T.~Gibbs, J.~Grilli, T.~Rogers, and S.~Allesina.
\newblock Effect of population abundances on the stability of large random
  ecosystems.
\newblock {\em Physical Review E}, 98(2):022410, 2018.

\bibitem{book-horn-johnson}
R.~A. Horn and C.~R. Johnson.
\newblock {\em Matrix analysis}.
\newblock Cambridge University Press, Cambridge, second edition, 2013.

\bibitem{kato2013perturbation}
T.~Kato.
\newblock {\em Perturbation theory for linear operators}, volume 132.
\newblock Springer Science \& Business Media, 2013.

\bibitem{leadbetter2012extremes}
M.~R. Leadbetter, G.~Lindgren, and H.~Rootz{\'e}n.
\newblock {\em Extremes and related properties of random sequences and
  processes}.
\newblock Springer Science \& Business Media, 2012.

\bibitem{may1972will}
R.~May.
\newblock Will a large complex system be stable?
\newblock {\em Nature}, 238(5364):413, 1972.

\bibitem{stone2018feasibility}
L.~Stone.
\newblock The feasibility and stability of large complex biological networks: a
  random matrix approach.
\newblock {\em Scientific reports}, 8(1):8246, 2018.

\bibitem{vivo2015large}
P.~Vivo.
\newblock Large deviations of the maximum of independent and identically
  distributed random variables.
\newblock {\em European Journal of Physics}, 36(5):055037, 2015.

\bibitem{wainwright2019high}
M.~J. Wainwright.
\newblock {\em High-dimensional statistics: A non-asymptotic viewpoint},
  volume~48.
\newblock Cambridge University Press, 2019.

\end{thebibliography}

\noindent {\sc Pierre Bizeul}\\
Ecole Normale Supérieure Paris-Saclay\\
61, avenue du Président Wilson\\
94235 Cachan Cedex\\
e-mail: {\tt pbizeul@ens-paris-saclay.fr}\\

\noindent {\sc Jamal Najim},\\
Laboratoire d'Informatique Gaspard Monge, UMR 8049\\
CNRS \& Universit\'e Paris Est Marne-la-Vall\'ee\\
5, Boulevard Descartes,\\
Champs sur Marne,
77454 Marne-la-Vall\'ee Cedex 2, France\\
e-mail: {\tt najim@univ-mlv.fr}\\

\end{document}